\documentclass[12pt]{amsart}

\usepackage{amssymb}
\usepackage{graphics}
\usepackage{latexsym}
\usepackage{amsmath}
\usepackage{amssymb}
\usepackage{amscd}
\usepackage{multicol}
\usepackage{multirow} 
\usepackage[cmtip,matrix,arrow]{xy}
\usepackage{tabularx}
\usepackage{makecell}

\newtheorem{thm}{Theorem}[section]
\newtheorem{prop}[thm]{Proposition}

\newtheorem{lem}[thm]{Lemma}

\theoremstyle{definition}
\newtheorem{dfn}[thm]{Definition}
\newtheorem{ex}[thm]{Example}

\theoremstyle{remark}
\newtheorem{rem}[thm]{Remark}

\newcommand{\RR}{\mathbb R}
\newcommand{\ZZ}{\mathbb Z}

\newcommand{\CC}{\mathbb C}

\newcommand{\PP}{\Delta}

\begin{document}



\title[Connectivity for real Grassmannians]{Connectivity properties of the Schur-Horn map for real Grassmannians}

\author[A.-L. Mare]{Augustin-Liviu ${\rm MARE}^1$}
{\footnote{Department of Mathematics and Statistics, University of Regina, Canada 

{\it Email address}: {\tt liviu.mare@gmail.com}}

\begin{abstract} To any  $V$ in the Grassmannian ${\rm Gr}_k({\mathbb R}^n)$ of $k$-dimensional vector subspaces in $\RR^n$
one can associate the diagonal entries
of the ($n\times n$) matrix corresponding to the orthogonal projection of $\RR^n$ to $V$. One obtains a map ${\rm Gr}_k(\RR^n)\to \RR^n$
(the Schur-Horn map). The main result of this paper is a criterion for pre-images of vectors in $\RR^n$ to be connected.
This will allow us to deduce  connectivity criteria for a certain class of subspaces of the real Stiefel manifold which arise naturally in frame theory. 
We extend in this way  results of Cahill, Mixon, and Strawn, see \cite{CMS}.  
\end{abstract}
\maketitle
 {\small 


}

\section{Introduction}\label{sect1}

Let ${\rm Gr}_k(\RR^n)$ be the Grassmannian of all $k$-dimensional linear subspaces of $\RR^n$. 
Equip $\RR^n$ with the Euclidean inner product. Attach to any $V$ in the Grassmannian the
linear endomorphism of $\RR^n$ given by orthogonal projection to $V$ and denote by $\pi_V$ 
its matrix relative to the canonical basis. In this way, one embeds ${\rm Gr}_k(\RR^n)$  into the space ${\rm Symm}_n(\RR)$ of symmetric 
$n\times n$ matrices, as follows:
$${\rm Gr}_k(\RR^n) = \{P \in {\rm Symm}_n(\RR) \mid P^2=P, \ {\rm tr}(P)=k\}.$$
Assign to any $n\times n$ matrix  $X$ the vector $\mu(X)$ which consists of the diagonal entries of $X$
and obtain in this way a map $\mu: {\rm Gr}_k(\RR^n)\to \RR^n$.
 
It is known that the image of $\mu$ is the polytope  $\PP_{n,k}$,  contained in the hyperplane of
equation $x_1+ \cdots + x_n=k$, with vertices at the points in $\RR^n$ whose coordinates
are 0 and 1, where 1 occurs exactly $k$ times. Alternatively, $\Delta_{n,k}$ is the hypersimplex which consists of all  
 $x=(x_1, \ldots, x_n) \in \RR^n$ such that $0\le x_i \le 1$ for all $1\le i \le n$ and $x_1+\cdots + x_n=k$. This description of the image of $\mu$
  is just a special case of a result which was 
proved by Schur and Horn, see \cite{Schur} and \cite{Ho}. (Even though  their result concerns originally 
matrices that are Hermitean, its version about symmetric matrices with real entries involves only
slight modifications, and deserves to be attributed to the same two authors, see \cite{Mi}, Thm.~2 and the note on p.~21, as well as \cite{MOA}, Ch.~9, Sect.~B, most notably
Thm.~B.2.)

We now consider  levels  of $\mu$, that is, pre-images of the form $\mu^{-1}(d)$, where $d=(d_1, \ldots, d_n) \in \PP_{n,k}$, and wonder 
whether they are connected subspaces of the Grassmannian. The question arises naturally in at least two branches of mathematics:
differential geometry and operator theory. We now explain the relevance for the first area and we will do the same for the second
 at the end of this section. 
The Schur-Horn theorem mentioned above is similar to the Atiyah-Guillemin-Sternberg convexity theorem
for Hamiltonian torus  actions,
see \cite{At} and \cite{Gu-St}. That is, both results claim that the image of a certain map is a  
polytope. For Hamiltonian torus actions this map is the moment map and an important feature of it 
is that its levels are all connected subspaces of the domain. 
Now, even though the image of the map $\mu$ above is a convex polytope,
its levels are in general not connected, as one can easily see
 in the case where $n=2$ and $k=1$, when $\mu$ is just the orthogonal projection of a circle to a straight line in the same plane.     
  However, by assuming that $2\le k \le n-2$, an interesting level $\mu^{-1}(d)$ which is connected arises, namely the one described by 
  $d_1= \cdots = d_n$: this was  shown by Cahill, Mixon, and Strawn, see  \cite[Cor.~1.3]{CMS}. Our first main result is as follows:

\begin{thm}\label{mainthm}  If $d=(d_1, \ldots, d_n)\in \PP_{n,k}$ is such that 
\begin{equation} \label{hypothesis} d_{i_1} +\cdots + d_{i_{n-k}}\ge 1 \ {\it 
 for \ any \ pairwise  \  distinct } \ i_1, \ldots, i_{n-k} \in \{1, \ldots, n\}, \end{equation}  then 
 $\mu^{-1}(d)$ is a connected subspace of ${\rm Gr}_k(\RR^n)$. 
\end{thm}

\begin{rem}
The condition described by eq.~(\ref{hypothesis}) is sufficient for the connectivity of $\mu^{-1}(d)$ without being
necessary.  To see this, consider again the special case $n=2$ and $k=1$: $\Delta_{2,1}$ is the straight line segment in 
$\RR^2$ of end-points $(1,0)$ and $(0,1)$, and the pre-image of  any of these two points is just a point.   Another relevant situation is presented in Example \ref{exa} below. 
\end{rem}

\begin{rem}
If the hypothesis of Thm.~\ref{mainthm} is satisfied, then $k \ge 2$ and $n-k \ge 2$.
Assume that $k=1$: if there exists a vector $d \in \PP_{n,k}$ with 
 $d_{i_1} +\cdots + d_{i_{n-1}}\ge 1$ for all pairwise distinct $i_1, \ldots, i_{n-1}\in \{1, \ldots, n\}$,
 by adding up all the inequalities above one obtains
 $$(n-1)(d_1+\cdots + d_n) \ge n,$$
 which contradicts $d_1+ \cdots + d_n =1$.
 If $k=n-1$ and there exists a vector $d \in \PP_{n,k}$ whose entries are all at least 1
 then their sum $d_1+ \cdots + d_n$ would be at least $n$; this is impossible, because $d$ is in $\PP_{n,k}$ and hence
 $d_1+ \cdots + d_n =n-1$.
\end{rem}

\begin{ex}\label{example1}
For any $n$ and $k$ such that $k\ge 2$ and $n-k\ge 2$, there is at least one $d\in \PP_{n,k}$ which satisfies 
assumption (\ref{hypothesis}). This is the vector whose entries are all equal to each other,
that is  $$d=\left(\frac{k}{n}, \ldots, \frac{k}{n}\right).$$
One needs to verify that
$$(n-k)\cdot \frac{k}{n} \ge 1.$$
But this is equivalent to
$$\frac{1}{k}+\frac{1}{n-k}\le 1,$$
which is obviously true. Thus Thm.~\ref{mainthm} is a generalization of  \cite[Cor.~1.3]{CMS}. 
\end{ex}

\begin{ex}\label{example2} More generally, one may take $d_1=\cdots=d_k=:\alpha$ and $d_{k+1} = \cdots = d_n =:\beta$
where $\alpha, \beta \ge 0$ such that
\begin{align*}
{}& k\alpha + (n-k)\beta = k\\
{}& \alpha \ge \beta\\
{}& (n-k)\beta \ge 1.
\end{align*}
From the first condition,
$$\alpha = 1-\frac{n-k}{k}\beta.$$
This is at least equal to $\beta$ if and only if
$$\beta \le \frac{k}{n}.$$
In conclusion, examples of vectors $d$ that are contained in $\PP_{n,k}$ and satisfy
the condition (\ref{hypothesis}) are determined by numbers $\beta$ which satisfy
$$\frac{1}{n-k}\le \beta \le \frac{k}{n}.$$ 
This class of examples is more general than the situation described in Ex.~\ref{example1},
except the case when $n=4$ and $k=2$. 
\end{ex}

In order to prove Thm.~\ref{mainthm} we will use the idea employed by Kirwan, see
\cite{Ki}, in dealing with  levels of moment maps for Hamiltonian group actions on symplectic manifolds. It all starts with the elementary observation
that the level $\mu^{-1}(d)$ we are interested in is the minimum level of  the real function
 given by \begin{equation} \label{vm} V\mapsto \| \mu(V) - d\|^2,\end{equation} $V\in {\rm Gr}_k(\RR^n)$. Without being a Morse function in the
 classical sense, this function gives rise
 to a Morse-type stratification of the Grassmannian, the strata being  embedded submanifolds.   
The proof of  Thm.~\ref{mainthm} relies on the crucial observation that, due to condition (\ref{hypothesis}), the stratification
contains no stratum of
codimension 1. Consequently, the 0-codimensional stratum is connected. Since this can be retracted
to the minimum level $\mu^{-1}(d)$,  the latter is connected as well. 

The stratification mentioned above along with its  properties needed in the proof is not immediately available. 
We obtain it as a particular case of a construction which belongs to Baird and Heydari, see
\cite{Ba}, in the general context of real loci of symplectic manifolds. More precisely,
the complex Grassmannian ${\rm Gr}_k(\CC^n)$ has a canonical symplectic structure and
contains ${\rm Gr}_k(\RR^n)$ as a real locus, see Sect.~\ref{sect:real} for more details.
It will turn out that height functions on the Grassmannian and Morse theoretical results concerning them
are essential ingredients: all that is nicely described in Guest's paper \cite{Gu}.

A final comment regarding Thm.~\ref{mainthm} concerns the relationship with the connectivity results for isoparametric submanifolds, which we have 
previously
obtained 
in \cite{Ma}. The real Grassmannian belongs to an isoparametric foliation of Euclidean space, which is just the orbit  
foliation of the action of the orthogonal group ${\rm O}(n)$ on ${\rm Symm}_n(\RR)$. However, one cannot use Thm.~1.1, or rather
Rem.~1.3 (a), in \cite{Ma} to prove our main result. The reason is that the root multiplicities of the isoparametric foliation at hand are all
equal to 1, whereas in \cite{Ma} they are supposed to be all at least 2.   
It could be possible to use the methods of our previous work to show  
that the function described by eq.~(\ref{vm}) is minimally degenerate in the sense of Kirwan and then
use this fact to construct the stratification of ${\rm Gr}_k(\RR^n)$ mentioned above.
We have not verified whether this approach really works, since it seems that using the results of Baird and Heydari's paper \cite{Ba} is more convenient. 

Pre-images of various vectors in $\PP_{n,k}$ are also relevant in frame theory. The last section of the paper is devoted to showing how to deduce from  Thm.~\ref{mainthm}  connectivity results for certain spaces of frames. To illustrate their relevance, we make a brief introduction to this topic.
For a more thorough presentation we refer to the monograph \cite{Wa}.   
 A (finite) {\it frame} is an   $n$-tuple $f_1, \ldots, f_n$ of vectors in $\RR^k$ for which  there exist two numbers $a$ and $b$ such 
that
\begin{equation}\label{defframe} a\|x\|^2 \le \sum_{i=1}^n  |\langle x, f_i\rangle |^2 \le b \|x\|^2,\end{equation}
for all $x\in \RR^n$. The inner product we used here is the standard Euclidean one and $\| \cdot\|$
is the corresponding norm. 
The notion described by the equation above was first considered in 1952 by Duffin and Schaeffer, see \cite{Du-Sch},
in an  infinite dimensional setting, namely the one of  Fourier series for functions in $L_2[-\pi, \pi]$ relative to a given spanning set which is not necessarily a basis.
 Afterwards it was noticed that  frames are relevant not only in harmonic analysis but also in other branches of  mathematics, such as signal processing \cite{Mallat}, coding \cite{Strohmer-Heath},
 wireless communication \cite{Strohmer}, etc.

A particular attention has received the situation when in eq.~(\ref{defframe}) one can choose $a=b$, in which case  we say that the resulting frame is {\it tight}.
  Under this assumption, after an obvious scalar multiplication/normalization of the
vectors $f_i$, one can achieve that $a=b=1$. A frame for which the latter condition is satisfied is said to be {\it normalized}. We will use the abbreviation 
NTF for {\it normalized tight frame}. Let us regard $f_i$ as column vectors and identify the $n$-tuple $(f_i)_{i=1}^n$ with the $k\times n$
matrix 
 $F=[ \ f_1 \ | \ \ldots \ | \ f_n \ ]$. This turns out to be convenient, since the NTF-assumption above is then equivalent to the simple equation  
\begin{align} {}& FF^t={\rm I}_k. \label{fft}  \end{align}
Yet another assumption often imposed in the literature concerns the norms of the vectors $f_i$, $1\le i \le n$. A frequent requirement is that
these are all equal. 
Efforts have been made to investigate the space of all equal-norm NTFs. 
It inherits from ${\rm Mat}_{k\times n}(\RR)$ the topology of subspace and is obviously a real algebraic variety, but its geometric structure is 
still far from being understood. For example, one could settle the question whether it is connected only recently, in the 2017 work \cite{CMS} by Cahill, Mixon and Strawn, which we already mentioned above.
They showed that the space of equal-norm NTFs is connected if and only if $2\le k \le n-2$, as previously conjectured in 2006 by Dykema and Strawn, see \cite{DS}.
It should be pointed out that the space which was actually studied in these two articles differs from ours by a normalization factor, 
due to the fact that in both places the right hand side of  eq.~(\ref{fft}) is $\frac{n}{k}I_k$ (consequently the norms of $f_i$ are all equal to 1
and the vectors determine what is called in \cite{DS} a ``spherical tight frame" and in \cite{CMS}  a ``unit tight frame").

As expected, even less is known under the assumption that the norms are prescribed, but not equal to each other. 
Concretely, let us pick a vector $d=(d_1, \ldots, d_n)$ whose components are all non-negative and impose the conditions
\begin{align} \|f_i\|^2=d_i, \ 1 \le i \le n.\label{norm} \end{align}
We call a frame $F$ which satisfies the assumptions (\ref{fft}) and (\ref{norm})  a $d$-NTF and denote  by ${\mathcal F}^d_{n,k}$ the space of all such frames.
The first natural question is now whether this space is non-empty. Necessary and sufficient conditions for this to occur were found in \cite{CL}, see also the beginning of  Sect.~\ref{app:frames} below.
In the spirit of \cite{CMS}, one may now wonder when is ${\mathcal F}^d_{n,k}$ connected. This question is addressed in the second half of this article.
Here is a brief summary. We first notice that eq.~(\ref{fft}) describes a submanifold of the space of $k\times n$ matrices, which is nothing but the real Stiefel manifold
${\rm V}_k(\RR^n)$. 
Thus  ${\mathcal F}_{n,k}^d$ is the subspace of ${\rm V}_k(\RR^n)$ described  by the equations (\ref{norm}). Furthermore, the canonical projection ${\rm V}_k(\RR^n)\to {\rm Gr}_k(\RR^n)$ induces an ${\rm O}(k)$-principal bundle 
${\mathcal F}_{n,k}^d \to \mu^{-1}(d_1, \ldots, d_n)$. If assumption (\ref{hypothesis}) is satisfied, the codomain of this map is a connected topological space.
The challenge is now to detect supplementary assumptions under which the domain of the map is connected as well. Such conditions can be found 
in Sect.~\ref{app:frames}, particularly  Propositions \ref{prop:first} and \ref{prop:second}. 

Note that this approach is not new, having been used 
     by Needham and Shonkwiler to prove connectivity results   for frames whose entries are complex numbers or quaternions, see  \cite{NS1} and \cite{NS2} respectively.
     It is worth pointing out that these results hold true without any assumption on $d$. This is because  {\it all} levels of $\mu$ are connected: in the complex case one applies the Atiyah-Guillemin-Sternberg  theorem for Hamiltonian group actions on symplectic manifolds, see \cite{At} and \cite{Gu-St}; in the quaternionic case one can rely on results concerning isoparametric
     submanifolds of Euclidean space obtained in \cite{Ma}. Besides, the principal bundle mentioned above has as structural group ${\rm U}(n)$ and ${\rm Sp}(n)$, respectively, which is in  both cases a connected group:
     thus the connectivity of $\mu^{-1}(d_1, \ldots, d_n)$ implies  automatically the connectivity of the space of frames.

      In
Subsections \ref{sub1} and \ref{sub2} we illustrate  Propositions \ref{prop:first} and \ref{prop:second}  by some examples of  spaces of normalized tight frames which are connected, without being of equal-norm type. 
Subsection \ref{poly} concerns connections with  spaces of polygons in $\RR^2$.   

\vspace{0.2cm} 

\noindent {\bf Acknowledgements.} I would like to thank Tom Needham and Clayton Shonkwiler for suggesting the topic of the present paper
and for a fruitful
exchange of ideas, Jost-Hinrich Eschenburg for helpful discussions, and the anonymous referee for carefully reading the manuscript and suggesting numerous improvements.  

\section{Height functions on Grassmannians} 

Equip the vector space ${\rm Symm}_n(\RR)$ with the inner product 
\begin{equation} \label{metric}\langle X, Y\rangle ={\rm tr} (XY), \end{equation} $X, Y\in {\rm Symm}_n(\RR)$. 
By means of the aforementioned embedding ${\rm Gr}_k(\RR^n)\subset {\rm Symm}_n(\RR)$, to any  $a=(a_1, \ldots, a_n) \in \RR^n$ one attaches the height function
$h_a: {\rm Gr}_k(\RR^n) \to \RR$, \begin{equation}\label{hadef}h_a(V)=\langle \pi_V,D_a\rangle,\end{equation}
where $D_a$ denotes the diagonal matrix ${\rm Diag} (a_1, \ldots, a_n)$.  
The following description of the critical set of $h_a$ has been obtained by Guest in\footnote{Guest is actually dealing with
the {\it complex} Grassmannian ${\rm Gr}_k(\CC^n)$; however he points out that similar results hold for the {\it real} Grassmannian,
see \cite[Remark, p.~170]{Gu}.}  \cite[Sect.~3.2]{Gu}
under the assumption that 
\begin{equation}\label{ineqa}a_1\ge a_2 \ge \cdots \ge a_n\ge 0.\end{equation}

Consider  the eigenspace decomposition of $D_a: \RR^n \to \RR^n$, which is
$$\RR^n=E_1 \oplus \cdots \oplus E_\ell.$$
To any subset $u$  with $k$ elements in $\{1, \ldots, n\}$, of the form $u=(u_1< \cdots < u_k)$, one attaches the vector space
$$V_u=\RR e_{u_1} \oplus \cdots \oplus \RR e_{u_k}$$
along with its orbit under the group ${\rm O}(E_1) \times \cdots \times  {\rm O}(E_\ell)$.
We say that two subsets $u$ and $u'$ with $k$ elements in $\{1, \ldots, n\}$ are equivalent if the
corresponding spaces $V_u$ and $V_{u'}$ are in the same orbit of the group mentioned above.
Let $[u]$ denote the equivalence class of $u$ and set  
$$M_{[u]}^a:={\rm O}(E_1) \times \cdots \times  {\rm O}(E_\ell) .V_u.$$
The latter is a product of Grassmannians, of the form
\begin{equation}\label{mu}M_{[u]}^a={\rm Gr}_{c_1}(E_1) \times \cdots \times {\rm Gr}_{c_\ell}(E_\ell),\end{equation}
where $c_i =\dim (V_u\cap E_i)$, $1\le i \le \ell.$
Note that $0\le c_i \le \dim E_i$, for all $1\le i \le \ell$
and $c_1 +\cdots + c_\ell=k$. It turns out that the critical set of $h_a$
is the union of all products of the form given by eq.~(\ref{mu}),  where $c_1, \ldots, c_\ell$ are
integers that satisfy the conditions above. Moreover, $h_a$ is a Morse-Bott function. 
Since ${\rm Gr}_k(\RR^n)$ is a submanifold of 
${\rm Symm}_n(\RR)$, it inherits a Riemannian metric from the inner product described by eq.~(\ref{metric}).
One can also describe the corresponding  stable manifolds of $h_a$ as follows:  to $M_{[u]}^a$
one associates
\begin{equation}\label{stablemnf}S_{[u]}^a = \{V \in {\rm Gr}_k(\RR^n) \mid \dim V \cap (E_1 \oplus \cdots \oplus E_i)= c_1+ \cdots + c_i, 
\ {\rm for \ all } \ 1\le i \le \ell\}.\end{equation}

Before recording all these facts, we would like to point out that the assumption
$a_n \ge 0$ in (\ref{ineqa}) can be dropped. Assume that $a_1\ge \cdots \ge a_n.$ Then
for any $V\in {\rm Gr}_k(\RR^n)$, 
\begin{align*}
 h_a(V)  & = {\rm tr}(\pi_V {\rm Diag}(a_1, \ldots, a_n))\\
{}& ={\rm tr}(\pi_V({\rm Diag}(a_1-a_n, \ldots, a_{n-1}-a_n, 0)+a_nI_n))\\
{}& ={\rm tr}(\pi_V({\rm Diag}(a_1-a_n, \ldots, a_{n-1}-a_n, 0))+{\rm tr}(a_n\pi_V)\\
{}&={\rm tr}(\pi_V({\rm Diag}(a_1-a_n, \ldots, a_{n-1}-a_n, 0))+ka_n.
\end{align*}
Thus up to an additive constant,  $h_a$ is just the height function corresponding to $(a_1-a_n, \ldots, a_{n-1}-a_n, 0)$,
where $a_1-a_n\ge  \cdots \ge a_{n-1}-a_n \ge 0$. Moreover, the eigenspace decompositions  of $\RR^n$ relative to 
${\rm Diag}(a_1, \ldots, a_n)$ and ${\rm Diag}(a_1-a_n, \ldots, a_{n-1}-a_n, 0)$ are the same.
We summarize the facts mentioned above:

\begin{thm}\label{thmguest} \cite[Sec.~3.2]{Gu}
Assume that $a=(a_1, \ldots, a_n)\in \RR^n$ satisfies $a_1\ge \cdots \ge a_n$. 
 Then the  height function $h_a: {\rm Gr}_k(\RR^n)\to \RR$ is Morse-Bott. Its critical set is the
 (disjoint) union of all submanifolds $M_{[u]}^a$, see (\ref{mu}),
 where $u$ is a subset with $k$ elements of $\{1, \ldots, n\}$.
For each such $u$, the stable manifold $S^a_{[u]}$ is given by (\ref{stablemnf}).
\end{thm} 

\begin{rem}\label{ambig}
The special situation when $a=(0, \ldots, 0)$ will play an important role later on. 
The function $h_a$ is in this case just the zero function so it is elementary that its critical set is the whole  ${\rm Gr}_k(\RR^n)$. 
Thm.~\ref{thmguest} leads to the same conclusion, as follows. 
Clearly $\ell=1$, $E_1=\RR^n$, and thus for any
subset $u\subset \{1, \ldots, n\}$ with $k$ elements, we have $c_1=\dim(V_u \cap E_1)=k$, which shows that both $M_{[u]}^{(0, \ldots, 0)}$ and 
$S_{[u]}^{(0, \ldots, 0)}$ are equal to the whole 
${\rm Gr}_k(\RR^n)$.     
\end{rem} 

We now adopt a general notation:  if $x=(x_1, \ldots, x_n)\in \RR^n$ and $\tau \in S_n$ is a permutation of
the set $\{1, \ldots, n\}$, then $x^\tau$ is the vector in $\RR^n$ given by
$$x^\tau_i=x_{\tau(i)},$$
for all $1\le i \le n.$  
This will be needed to achieve our  next goal, which is to take $a=(a_1, \ldots, a_n) \in \RR^n$ {\it arbitrary} and determine the critical set of the function $h_a$.
First reorder the entries of $a$ and get  $a^\sigma\in \RR^n$ such that
\begin{equation}\label{orfera}a_{\sigma(1)} \ge \cdots \ge a_{\sigma(n)},\end{equation}
where $\sigma \in S_n$.
Consider the linear (actually orthogonal) transformation $g$ of $\RR^n$ given by
\begin{equation}\label{gfera} g(e_i)=e_{\sigma(i)},\end{equation}
for all $1\le i \le n$. The obvious induced automorphism of ${\rm Gr}_k(\RR^n)$ is denoted by $g$ as well. 

\begin{prop}\label{critgen}
\begin{itemize}
\item[(i)] For any $V\in {\rm Gr}_k(\RR^n)$, one has 
\begin{equation} \label{ha} h_{{a}^\sigma}(V)=h_a(gV)\end{equation}
and
\begin{equation} \label{mug} \mu(gV)=\mu(V)^{\sigma^{-1}}.\end{equation}
\item[(ii)] 
The function $h_a$ is Morse-Bott and its critical set  is  
$${\rm Crit}(h_a)=g{\rm Crit}(h_{{a}^\sigma})=\bigcup gM_{[u]}^{{a}^\sigma},$$
where ${u\subset \{1,\ldots, n\}}$ with $k$ elements.  
Furthermore, the  stable manifold corresponding to $gM_{[u]}^{{a}^\sigma}$ is $gS_{[u]}^{{a}^\sigma}$, where
$S_{[u]}^{{a}^\sigma}$ is described by eq.~(\ref{stablemnf}). 
\end{itemize}
\end{prop}

\begin{proof} (i) We first note that the matrices of the projections $\pi_V$ and $\pi_{gV}$
are related by $\pi_{gV}=g\pi_V g^{-1}$. Consequently,
\begin{align*}
 h_a(gV)  & = {\rm tr}(\pi_{gV} D_a)\\
{}& ={\rm tr}(g\pi_{V} g^{-1} D_a)\\
{}& ={\rm tr}(\pi_{V} g^{-1} D_ag)\\
{}&={\rm tr}(\pi_{V}  D_{a^\sigma})\\
{}& = h_{{a}^\sigma}(V). 
\end{align*}

We now prove (\ref{mug}):

\begin{align*}
 \mu(gV)  & = \mu(g\pi_V g^{-1})\\
{}& =(\ldots, \langle  g\pi_V g^{-1}(e_i), e_i \rangle, \ldots ) \\
{}& =(\ldots, \langle \pi_V g^{-1}(e_i), g^{-1}(e_i) \rangle, \ldots )\\
{}&=(\ldots, \langle \pi_V (e_{{\sigma^{-1}i}}), e_{{\sigma^{-1}i}} \rangle, \ldots)\\
{}&=\mu(V)^{\sigma^{-1}}.
\end{align*}

Item (ii)  is a straightforward consequence of eq.~(\ref{ha}).

\end{proof} 

\begin{rem}\label{ambig2} 
For a given $a\in \RR^n$ there are in general several ways to choose $\sigma$ which lead 
to the same $a^\sigma$, where the latter satisfies the condition  (\ref{orfera}). However, one can easily see that the
corresponding 
$gM_{[u]}^{a^\sigma}$ and $gS_{[u]}^{a^\sigma}$ are independent of the choice of $\sigma$ (and implicitly of $g$). 
\end{rem}

\section{The Grassmannian ${\rm Gr}_k(\RR^n)$ as a real locus} \label{sect:real}

In this section we  collect results from \cite{Gu} and \cite{Ba} which will
be used in the proof of Thm.~\ref{mainthm}. 
We  need to involve the complex Grassmannian
${\rm Gr}_k(\CC^n)$, whose elements are the complex $k$-dimensional vector subspaces of $\CC^n$.
Just like its real version, it is convenient to embed it into the space of
complex Hermitean $n\times n$ matrices, via the assignment
${\rm Gr}_k(\CC^n) \ni V \mapsto \pi_V$, the latter being the linear endomorphism of $\CC^n$
given by projecting $\CC^n$ to $V$ orthogonally relative to the standard Hermitean inner product of $\CC^n$.  
Consider the torus 
$$T^n:=\{ {\rm Diag}(z_1, \ldots, z_n) \mid z_j \in \CC, |z_j|=1, j=1, \ldots, n \},$$
along with its canonical action on $\CC^n$ by $\CC$-linear automorphisms and the induced action on the complex Grassmannian.
Another way to see this is by considering the action of $T^n$ on the space of Hermitean $n\times n$ matrices
given by matrix conjugation and restrict it to ${\rm Gr}_k(\CC^n)$, this being an invariant subspace. 
The complex Grassmannian has a canonical symplectic structure and the action of $T^n$ described above turns out to be Hamiltonian.
 A moment map is $\mu: {\rm Gr}_k(\CC^n) \to \RR^n$,
 which associates to $V$ in the domain the diagonal component of $\pi_V$, as explained for instance in \cite[Sect.~4]{Gu}.
 Note that ${\rm Gr}_k(\RR^n)$ is contained in ${\rm Gr}_k(\CC^n)$ in such a way that
 the map $\mu$ on the former space which was defined in Sect.~\ref{sect1}  is the
 restriction of the aforementioned  moment map: this is why denoting them both by
 $\mu$ should not create any confusion. 

The embedding ${\rm Gr}_k(\RR^n) \subset {\rm Gr}_k(\CC^n)$ is for us an essential ingredient.
It actually arises as the ``real locus" of the complex Grassmannian relative to
the natural complex structure. More specifically, the complex conjugation map
$\tau: \CC \to \CC$, $\tau(z)=\bar{z}$, for $z\in \CC$, induces an involutive automorphism
of ${\rm Gr}_k(\CC^n)$ which is also denoted by $\tau$ and whose fixed point set is 
${\rm Gr}_k(\RR^n)$.  
 (We refer to \cite[Example 2.2]{Ba} for more  details.)  
This fact has important consequences. For example, the images of both the complex and the
real Grassmannians under $\mu$ are equal, i.e.~$\mu({\rm Gr}_k(\RR^n)) =\mu ({\rm Gr}_k(\CC^n))=\PP_{n,k}$.
  
As indicated in Thm.~\ref{mainthm}, we are interested in the preimage under $\mu$ of a certain vector $d\in \PP_{n,k}$.   
 Consider the function
 $f:{\rm Gr}_k(\CC^n) \to \RR$, $$f(V) = \|\mu(V)-d \|^2, \ V\in {\rm Gr}_k(\CC^n),$$
 where the norm is the one arising from the standard inner product on $\RR^n$.  
 The restriction of $f$ to ${\rm Gr}_k(\RR^n)$ is denoted by 
 $f_{\rm r}$.  To describe the critical set of $f_{\rm r}$ and the induced stratification of the real Grassmannian, we need the height functions 
 $h_a:  {\rm Gr}_k(\RR^n) \to \RR$, $a\in \RR^n$, defined by eq.~(\ref{hadef}) and the description of ${\rm Crit}(h_a)$ given by Prop.~\ref{critgen} (ii). 
 
\begin{thm}\label{stratif} \begin{itemize} 
\item[(i)] The critical set of $f_{\rm r}$ is the union of all spaces of the form
$\mu^{-1}(d+a) \cap {\rm Crit}(h_a)$ where $a\in \RR^n$. 
There are only finitely many $a \in \RR^n$ and $u\subset \{1, \cdots, n\}$ with $k$ elements such that
$\mu^{-1}(d+a) \cap gM_{[u]}^{{a}^\sigma}$ is non-empty. For any such $a$ and $u$, there is a submanifold
$\Sigma_{a, [u]}$  of ${\rm Gr}_k(\RR^n)$ which contains $\mu^{-1}(d+a) \cap gM_{[u]}^{{a}^\sigma}$
as a deformation retract. These submanifolds induce a stratification 
\begin{equation}\label{stratsigma}{\rm Gr}_k(\RR^n)= \bigsqcup \Sigma_{a, [u]}.\end{equation}
The codimension of $\Sigma_{a,[u]}$ in
${\rm Gr}_k(\RR^n)$ is equal to the index of $h_a$ along $gM_{[u]}^{{a}^\sigma}$, that is, the codimension of
$gS_{[u]}^{a^\sigma}$.
\item[(ii)] There is exactly one stratum $\Sigma_{a,[u]}$ of codimension zero, namely 
the one corresponding to $a=(0, \ldots, 0)$ 
and $u$  any subset with $k$ elements of\footnote{Cf.~also Rem.~\ref{ambig}.}
$\{1, \ldots, n\}$.  Moreover,  $M_{[u]}^{(0, \ldots, 0)}$ is the whole ${\rm Gr}_k(\RR^n)$ and hence 
$\Sigma_{(0, \ldots, 0),[u]}$ contains $\mu^{-1}(d)$ as a deformation retract. 
\end{itemize} 
\end{thm} 
 
\begin{proof} 
(i) The critical set of $f:{\rm Gr}_k(\CC^n) \to \RR$ can be described using general 
formulas obtained by Kirwan in \cite{Ki}, which are also presented in \cite[Sect.~3.1]{Ba}. 
That description relies on the knowledge of  the critical set of
the height functions on the {\it complex} Grassmannian $ {\rm Gr}_k(\CC^n)$. 
As already mentioned, this Grassmannian admits a canonical embedding into the space of
Hermitean matrices. The latter being equipped with the inner product described by eq.~(\ref{metric}),
one attaches to any $a\in \RR^n$ the height function defined by eq.~(\ref{hadef}), where $V$ is this time in 
the complex Grassmannian. 
Since this new function is an extension of the height function on the real Grassmannian, we will denote it again by
$h_a$. 
The critical set of $h_a: {\rm Gr}_k(\CC^n) \to \RR$ is described in \cite[Sect.~3]{Gu}. The formula is identical\footnote{We actually rely on the remark
made in \cite{Gu} on page 170.} to the
one in Prop.~\ref{critgen} (ii), where $\sigma$ and $g$ are  determined by $a$ by means of equations (\ref{orfera}) and  (\ref{gfera}).    
In turn, $M_{[u]}^{a^\sigma}$ is described by eq.~(\ref{mu}), with the only difference  that this time
$E_1, \ldots, E_\ell$ should be replaced by their complexifications
$E_1^c, \ldots, E_\ell^c$, which are complex vector subspaces of $\CC^n$, and 
${\rm Gr}_{c_1}(E_1^c), \ldots, {\rm Gr}_{c_\ell}(E_\ell^c)$ are complex Grassmannians.
Note that the intersection of the  product ${\rm Gr}_{c_1}(E_1^c)\times  \cdots \times  {\rm Gr}_{c_\ell}(E_\ell^c)$
with ${\rm Gr}_k(\RR^n)$ is just ${\rm Gr}_{c_1}(E_1)\times  \cdots \times  {\rm Gr}_{c_\ell}(E_\ell)$.

Now turning to the critical set of $f:{\rm Gr}_k(\CC^n) \to \RR$, by \cite[Lemma 3.12]{Ki} (cf.~also \cite[Sect.~3.1]{Ba}), this is the union of all non-empty
intersections $\mu^{-1}(d+a) \cap {\rm Crit}(h_a:{\rm Gr}_k(\CC^n) \to \RR)$, for  $a\in \RR^n$.
At the same time, by \cite[Prop.~3 (i)]{Ba}, the critical set of $f_{\rm r}:{\rm Gr}_k(\RR^n) \to \RR$
is equal to ${\rm Crit}(f)\cap {\rm Gr}_k(\RR^n)$. The description of ${\rm Crit}(f_{\rm r})$ stated above is now clear.
The finiteness claim follows from \cite[item K1 in Sect.~3.1]{Ba}.

 To prove the remaining part of item (i), one uses the results of \cite[Sect.~3]{Ki} for the moment map $\mu: {\rm Gr}_k(\CC^n) \to \RR^n$ 
to construct a Morse stratification for $f$. By intersecting each stratum with ${\rm Gr}_k(\RR^n)$
one obtains the desired stratification, see  
\cite[Prop.~3 (ii)]{Ba}. 
To prove the claim about the codimension of $\Sigma_{a,[u]}$ one uses
 \cite[Prop.~3 (iv)]{Ba}, \cite[item K3 in Sect.~3.1]{Ba}, and 
 the  fact that for a fixed $a\in \RR^n$, the index of $h_a : {\rm Gr}_k(\CC^n)\to \RR$
 along any of its critical manifolds is twice  the index of $h_a : {\rm Gr}_k(\RR^n)\to \RR$
 along  the real form of that critical manifold  (see the discussion above concerning the two critical sets). 
 The last claim follows from the remark made on page 170 in \cite{Gu}.
 
 (ii) In the Morse stratification constructed by Kirwan there is exactly one stratum of codimension zero
 in ${\rm Gr}_k(\CC^n)$, 
 namely the one corresponding to the minimum level of $f:{\rm Gr}_k(\CC^n) \to \RR$, which is the preimage 
 of $d$ under $\mu:{\rm Gr}_k(\CC^n) \to \RR^n$. Its intersection with ${\rm Gr}_k(\RR^n)$ is therefore the only
 stratum in (\ref{stratsigma}) of codimension zero in
 ${\rm Gr}_k(\RR^n)$, by \cite[Prop.~3 (iv)]{Ba}.  
 To identify the corresponding $a$ and $u$, one takes into account that the
 corresponding critical set is contained in $\mu^{-1}(d+a)$ and thus 
 the levels $\mu^{-1}(d)$ and $\mu^{-1}(d+a)$ have common points, which implies  $a=(0,\ldots, 0)$. 
\end{proof}

 \section{Proof of the main result}

The key result  is:

\begin{prop} \label{prop1ton}
Assume $d\in \PP_{n,k}$ satisfies the assumption (\ref{hypothesis}). If $a\in \RR^n$ and $u\subset \{1, \ldots, n\}$ with $k$ elements are such that 
$\mu^{-1}(d+a) \cap g M_{[u]}^{{a}^\sigma}$ 
is non-empty and the codimension of $\Sigma_{a,[u]}$ is strictly greater than 0
then the codimension of $\Sigma_{a,[u]}$  is at least equal to 2.
\end{prop}

 \begin{proof}
 Assume that the codimension of $\Sigma_{a,[u]}$ is equal to 1. 
 By Thm.~\ref{stratif} (i), the codimension of   $S_{[u]}^{{a}^\sigma}$ is equal to 1 as well.
 Write $a^\sigma$  as   
\begin{equation}\label{absigm}
(a_{\sigma(1)}, \ldots, a_{\sigma(n)})=(\underbrace{b_1, \ldots, b_1}_{m_1 },
\ldots, \underbrace{b_\ell, \ldots, b_\ell}_{m_\ell })\end{equation}
where $b_1> \cdots > b_\ell$.  
 Let $$\RR^n=E_1\oplus \cdots \oplus E_\ell$$
be the eigenspace decomposition of ${\rm Diag}(a_{\sigma(1)}, \ldots, a_{\sigma(n)})$,
where $\dim E_i=m_i$, $1\le i \le \ell$.
Obviously
\begin{align*}
{}& E_1 = {\rm Span} \{e_1, \ldots, e_{m_1}\},\\
{}& E_2 = {\rm Span} \{e_{m_1+1}, \ldots, e_{m_1+m_2}\},\\
{}& {\rm etc.}
\end{align*}
Recall that   
\begin{equation}\label{mva} M_{[u]}^{{a}^\sigma}={\rm Gr}_{c_1}(E_1) \times  \cdots \times {\rm Gr}_{c_\ell}(E_\ell),\end{equation}
 for some $0 \le c_i \le m_i$ with $c_1 +\cdots + c_\ell =k$. Also, the corresponding $S_{[u]}^{{a}^\sigma}$ is described by eq.~(\ref{stablemnf}).

 \noindent{ \it Claim.} There exists $r$ with $1 \le r \le \ell$ such that
 \begin{align}\label{c1}
 \left( c_1 + \cdots + c_r\right)\left(n-k-(m_1+ \cdots + m_r)+ c_1 + \cdots + c_r\right)=1. 
 \end{align}  
 
 To justify this claim, for each $1\le i \le \ell$ one considers 
 $$S_i:=\{V \in {\rm Gr}_k(\RR^n) \mid \dim V \cap (E_1\oplus \cdots \oplus E_i) =c_1 + \cdots +c_i\},$$
 which, according to~\cite[Thm.~4.1 (c)]{Wo}, is a  submanifold of 
 ${\rm Gr}_k(\RR^n)$ of codimension equal to $(c_1+\cdots + c_i)(n-k-(m_1+ \cdots + m_i)+c_1+\cdots + c_i)$.
Consider the sequence of inclusions
 $$S_1 \supseteq S_1 \cap S_2 \supseteq \cdots \supseteq S_1 \cap S_2 \cap \cdots \cap S_\ell.$$
 Observe that the last element of the chain above is just $S_{[u]}^{{a}^\sigma}$, which, as already mentioned, is a one-codimensional submanifold
 of  ${\rm Gr}_k(\RR^n)$. Let now $r$ be the smallest integer, $1\le r \le \ell$,  with the property that
 $S_1\cap \cdots \cap S_r$ is a one-codimensional submanifold of ${\rm Gr}_k(\RR^n)$. More precisely, if $r>1$ then for all $1 \le i \le r-1$, the space
 $S_1\cap \cdots \cap S_i$ is zero-codimensional, thus open in ${\rm Gr}_k(\RR^n)$.   Finally, take into account the formula for the codimension of 
 $S_r$ mentioned above. The claim is proved.

 Now eq.~(\ref{c1}) implies that both factors in the left hand side are equal to 1, thus
\begin{align*}
{}& c_1 + \cdots + c_r=1 \ {\rm and} \\
{}& m_1+\cdots + m_r=n-k.
\end{align*} 
Also observe that $r$ cannot be equal to $\ell$, since
 $c_1 + \cdots + c_\ell =k$, which is at least 2. From the previous two equations, 
\begin{align*}
{}& c_{r+1} + \cdots + c_\ell=k-1 \ {\rm and} \\
{}& m_{r+1}+\cdots + m_\ell=k.
\end{align*}

 On the other hand,  the condition  $\mu^{-1}(d+a) \cap g M_u^{{a}^\sigma}\neq \emptyset$ is equivalent to
 $d+a \in \mu(g M_u^{{a}^\sigma})$ and then, by means of eq.~(\ref{mug}), to
 $$d^\sigma+a^\sigma\in \mu( {\rm Gr}_{c_1}(E_1) \times  \cdots \times {\rm Gr}_{c_\ell}(E_\ell)).$$
Hence
 \begin{equation}\label{asigma}d^\sigma+a^\sigma = \mu(V) \ {\rm where} \ V=V_1\oplus \cdots \oplus V_\ell,\end{equation}
$V_i$ being a linear subspace of $E_i$ of dimension $c_i$, for all $1\le i \le r$.  
 But then $\pi_V$ decomposes as the direct sum of the  maps
 $\pi_{V_i}: E_i \to E_i$, $1\le i\le \ell$. Consequently,
 $\mu(V)$ is a vector with $n$ components which are 
 the diagonal entries of the matrix of $\pi_{V_1}$ relative to the basis $e_1, \ldots, e_{m_1}$, followed by the diagonal entries of 
 the matrix of $\pi_{V_2}$ relative to the basis $e_{m_1+1}, \ldots, e_{m_1+m_2}$ etc.
 By equations (\ref{asigma}) and (\ref{absigm}), the diagonal of the matrix of $\pi_{V_1}$ is $(d^\sigma_1+b_1, \ldots d^\sigma_{m_1}+b_1)$. 
 But  $V_1\in {\rm Gr}_{c_1}(E_1)$, hence the trace of $\pi_{V_1}$ is $c_1$, that is  
 $d^\sigma_1+b_1+ \cdots +d^\sigma_{m_1}+b_1=c_1$, which gives:
 \begin{equation}\label{sigmam0}d^\sigma_1+ \cdots +d^\sigma_{m_1}+m_1b_1=c_1.\end{equation}
 Similarly,
 \begin{align}
 {}& \nonumber d^\sigma_{m_1+1}+ \cdots +d^\sigma_{m_1+m_2}+m_2b_2=c_2\\
 {}&\nonumber  \cdots\\
 {}& \label{sigmam2} d^\sigma_{m_1+ \cdots + m_{r-1}+1} + \cdots + d^\sigma_{m_1+ \cdots + m_{r-1}+m_r}+m_rb_r=c_r.
 \end{align}
 By adding up the last equations side by side and taking into account that $m_1+ \cdots + m_{r}=n-k$ and 
 $c_1+ \cdots + c_r=1$, one obtains
 \begin{equation}\label{dsigma1}d^\sigma_1 + \cdots + d^\sigma_{n-k} +m_1b_1 + \cdots + m_rb_r=1.\end{equation} 
 The list of equations (\ref{sigmam0}) - (\ref{sigmam2}) can now be continued.
 By adding up the remaining equations side by side and taking into account that
$m_{r+1} + \cdots + m_\ell=k$ and
$c_{r+1} + \cdots + c_\ell = k-1$, one obtains
\begin{equation}\label{dsigma2}d^\sigma_{n-k+1} + \cdots + d^\sigma_{n}+ m_{r+1}b_{r+1} + \cdots + m_\ell b_\ell=k-1.\end{equation}

We now prefer to rewrite (\ref{dsigma1}) and (\ref{dsigma2}) as
\begin{align*}
{} & m_1b_1 + \cdots + m_rb_r=1-(d^\sigma_1 + \cdots + d^\sigma_{n-k})\\
{}& m_{r+1}b_{r+1} + \cdots + m_\ell b_\ell=k-1- (d^\sigma_{n-k+1} + \cdots + d^\sigma_{n}).
\end{align*} 
 Recall that $b_1>b_2> \cdots > b_\ell$. This implies that
 $$\frac{m_1b_1 + \cdots + m_rb_r}{m_1+ \cdots +m_r}\ge b_r > b_{r+1} \ge \frac{m_{r+1}b_{r+1} + \cdots + m_\ell b_\ell}{m_{r+1}+ \cdots +m_\ell}$$
 and thus
 $$\frac{1-(d^\sigma_1 + \cdots + d^\sigma_{n-k})}{n-k} > \frac{ k-1- (d^\sigma_{n-k+1} + \cdots + d^\sigma_{n})}{k}.$$
 On the other hand,
 $d^\sigma_1 + \cdots + d^\sigma_{n}=k$, hence $d^\sigma_{n-k+1} + \cdots + d^\sigma_{n}=k-(d^\sigma_1 + \cdots + d^\sigma_{n-k})$. 
 By substituting this in the previous inequality one immediately obtains
 $$d^\sigma_1 + \cdots + d^\sigma_{n-k}< 1.$$
 But this contradicts the assumption on $d$ made by eq.~(\ref{hypothesis}). This finishes the proof. 
\end{proof}

The proof of  our main result is now straightforward.

\vspace{0.2cm}

\noindent{\it Proof of Theorem \ref{mainthm}.} By Prop.~\ref{prop1ton}, all strata in the stratification (\ref{stratsigma}) 
are submanifolds of codimension at least 2, except $\Sigma_{(0,\ldots, 0),[u]}$. It follows that the latter is 
connected. At the same time, this stratum contains $\mu^{-1}(d)$ as a deformation retract, see Thm.~\ref{stratif} (ii).
Thus $\mu^{-1}(d)$ is connected as well. $\square$

\section{Two examples} 
We will be looking at ${\rm Gr}_2(\RR^4)$ with two choices of $d$ in the corresponding polytope $\PP_{4,2}$. 

\begin{ex} \label{exa} Let us first take $d=(1,1,0,0)$. Its pre-image under $\mu$ consists of all
2-planes $V$ in $\RR^4$ such that
\begin{align*}
{}& \langle \pi_V(e_1), e_1 \rangle =1\\
 {}& \langle \pi_V(e_2), e_2 \rangle =1\\
 {}& \langle \pi_V(e_3), e_3 \rangle =0\\
 {}& \langle \pi_V(e_4), e_4 \rangle =0
 \end{align*}
The first two conditions show that $e_1$ and $e_2$ are in $V$, in other words,  that $V$ is the span of $e_1$ and $e_2$.
So in this case $\mu^{-1}(d)$ is just a point in ${\rm Gr}_2(\RR^4)$, which is a connected subspace.
Note however that $d$ does not satisfy  assumption (\ref{hypothesis}).
\end{ex} 

\begin{ex} We now take $d=(1,\frac{1}{3},\frac{1}{3},\frac{1}{3})$. This time the 2-planes $V$ in $\mu^{-1}(d)$ are determined by:
\begin{align*}
{}& \langle \pi_V(e_1), e_1 \rangle =1\\
 {}& \langle \pi_V(e_2), e_2 \rangle =\frac{1}{3}\\
 {}& \langle \pi_V(e_3), e_3 \rangle =\frac{1}{3}\\
 {}& \langle \pi_V(e_4), e_4 \rangle =\frac{1}{3}
 \end{align*}
 From the first condition, $e_1\in V$. So $V$ is uniquely determined by its quotient $V/ \RR e_1$,
 which is a line, say $\ell$, in $\RR^4/ \RR e_1 \simeq \RR^3$ such that
 \begin{align*}
 {}& \langle \pi_\ell(e_2), e_2 \rangle =\frac{1}{3}\\
 {}& \langle \pi_\ell(e_3), e_3 \rangle =\frac{1}{3}\\
 {}& \langle \pi_\ell(e_4), e_4 \rangle =\frac{1}{3}.
 \end{align*}
It is an easy exercise to see that there are only four lines $\ell$ in $\RR^3$ which satisfy these conditions, namely $\RR( e_2\pm e_3 \pm e_4)$
(cf.~also \cite[Thm.~3.1]{DS}). 
So this time $\mu^{-1}(d)$ is not connected. As expected,
  assumption (\ref{hypothesis}) is  not satisfied by $d$.  
\end{ex}

\section{Application: connectivity of the space of real tight frames} \label{app:frames}
Recall that if $1\le k \le n$ are two integers then the Stiefel manifold ${\rm V}_k(\RR^n)$ is the space of all ordered sets of $k$ orthonormal vectors in
$\RR^n$. Out of any such set one produces a matrix, as follows: the rows are the coordinates of the vectors relative to the canonical basis of $\RR^n$.
In this way, ${\rm V}_k(\RR^n)$ becomes a submanifold of the real vector space ${\rm Mat}_{k\times n}(\RR)$ of $k\times n$ matrices with real entries. 
This is a consequence of the pre-image theorem in differential geometry, since a $k\times n$ matrix $F$ belongs to
 ${\rm V}_k(\RR^n)$ exactly if $$FF^t={\rm I}_k,$$ where the superscript $t$ indicates the matrix transposition. 
 
 We will also be interested in the map ${\rm V}_k(\RR^n)\to {\rm Gr}_k(\RR^n)$, which attaches to any $k$-tuple in the domain the linear subspace
 of $\RR^n$ spanned by it.
The map can be nicely expressed if we regard its domain and codomain as subspaces of   ${\rm Mat}_{k\times n}(\RR)$ and ${\rm Mat}_{n\times n}(\RR)$, respectively:
it is of the form \begin{equation}\label{map}F \mapsto F^tF,\end{equation} for any $F \in {\rm V}_k(\RR^n)$.
This map is a principal ${\rm O}(k)$-bundle. Here ${\rm O}(k)$ acts on ${\rm Mat}_{k\times n}(\RR)$ by matrix multiplication from the left,
and this action leaves ${\rm V}_k(\RR^n)$ invariant. 
Say now that the columns of an arbitrary $F \in {\rm Mat}_{k\times n}(\RR)$ are $f_1, \ldots, f_n$, which we indicate as
$F=[ \ f_1  \ | \ \ldots \ | \ f_n \ ]$. One is wondering what are the possible values of
$$d_1:=\|f_1\|^2, \ldots, d_n :=\|f_n\|^2,$$
when $F\in  {\rm V}_k(\RR^n)$. To answer this question, simply observe that these numbers are just the diagonal entries of the matrix $F^tF$, which is in
${\rm Gr}_k(\RR^n)$. Thus, a vector $d=(d_1, \ldots, d_n)$ arises in the way indicated above if and only if
it belongs to the hypersimplex $\Delta_{k,n}$.

Let us now fix $d=(d_1, \ldots, d_n)\in \Delta_{k,n}$. Recall that a $d$-NTF is a matrix
$$ F=[ \ f_1  \ | \ \ldots \ | \ f_n \ ] \in {\rm V}_k(\RR^n)$$
such that $\|f_j\|^2=d_j, {\rm \ for \ all \ } 1\le j \le n$. 
We denoted by ${\mathcal F}_{n,k}^d$ the collection of all such matrices $F$ and  endowed this space with the topology of subspace 
of ${\rm Mat}_{k \times n}(\RR)$.

\begin{dfn}\label{admis} We say that $d=(d_1, \ldots, d_n)\in \Delta_{n, k}$ is $k$-{\it admissible} if ${\mathcal F}_{n,k}^d$  is connected.\end{dfn}

 As mentioned in the introduction, the $n$-tuple $\left(\frac{k}{n}, \ldots, \frac{k}{n} \right)$ 
 is $k$-admissible if and only if $2\le k \le n-2$, as shown by  Cahill, Mixon, and Strawn, see \cite[Thm.~1.2]{CMS}.  
 The goal of this section is to present other values of $(d_1, \ldots, d_n)$  which are $k$-admissible.


We first make the connection  with the level sets addressed in the previous sections, by means of principal ${\rm O}(k)$-bundle
$ {\rm V}_k(\RR^n) \to {\rm Gr}_k(\RR^n)$ given by eq.~(\ref{map}). The image of $ {\mathcal F}_{n,k}^d$ is clearly 
just $\mu^{-1}\left(d_1, \ldots, d_n\right)$. Since $ {\mathcal F}_{n,k}^d$ is an ${\rm O}(k)$-invariant subspace of ${\rm V}_k(\RR^n)$,  one obtains a
 homeomorphism $$ {\mathcal F}_{n,k}^d/{\rm O}(k) \simeq \mu^{-1}\left(d_1, \ldots, d_n\right),$$
 where the quotient in the left hand side is relative to the  action of ${\rm O}(k)$  by left multiplication on 
  ${\rm Mat}_{k\times n}(\RR)$. 
  
{\bf Throughout the rest of this section we will assume that $d$ satisfies assumption (\ref{hypothesis}).}

 Consequently, by Thm.~\ref{mainthm}, the quotient $ {\mathcal F}_{n,k}^d/{\rm O}(k)$ is connected. But ${\rm O}(k)/{\rm SO}(k) \simeq \ZZ_2$, thus  the canonical map
 ${\mathcal F}_{n,k}^d/{\rm SO}(k) \to {\mathcal F}_{n,k}^d/{\rm O}(k)$ is a double covering. Since ${\rm SO}(k)$ is connected,
 it follows that  
 ${\mathcal F}_{n,k}^d$ has one or two connected components. Consequently, we have: 
 
\begin{lem} \label{lemeq}
The following conditions are equivalent:

(a) The space  ${\mathcal F}_{n,k}^d$ is connected.

(b) The space  ${\mathcal F}_{n,k}^d$ is path connected.

(c)  There exist  $F\in {\mathcal F}_{n,k}^d$,  $A \in {\rm O}(k)$ with determinant equal to $-1$, and
 a continuous path in ${\mathcal F}_{n,k}^d$ from $F$ to $AF$.
  

\end{lem}

\begin{proof} The only assertion which still needs explanations concerns item (b): the main point is that ${\mathcal F}_{n,k}^d$ is a real algebraic variety and is therefore 
locally path connected, by a theorem of \L{}ojasiewicz \cite{Loj}.
\end{proof} 
  
Our goal is to identify classes of admissible vectors. The following two lemmas will be useful. The first one is a generalization of \cite[Prop.~7.2]{DS}.
 
 \begin{lem} \label{lem:DS} If ${\mathcal F}_{n,k}^d$ satisfies condition (c) in Lemma \ref{lemeq}, then 
 ${\mathcal F}_{n,n-k}^{1-d}$ satisfies it too. Here by $1-d$ we denoted the sequence\footnote{Note that
 $d\in \Delta_{n,k}$ if and only if $1-d\in \Delta_{n,n-k}$.} $(1-d_1, \ldots, 1-d_n)$.  
 \end{lem}
 
 \begin{proof} First observe that the Grassmannians ${\rm Gr}_k(\RR^n)$ and ${\rm Gr}_{n-k}(\RR^n)$
 are both submanifolds of ${\rm Symm}_n(\RR)$ and the map ${\rm Gr}_k(\RR^n)\to {\rm Gr}_{n-k}(\RR^n)$,
 $$P\mapsto {\rm I}_n-P,$$
 $P\in {\rm Gr}_k(\RR^n)$, is a diffeomorphism. Let $\nu:{\rm Gr}_{n-k}(\RR^n) \to \RR^n$ be the projection to the diagonal entries.
 Relevant for us is the pre-image $\nu^{-1}\left(1-d_1, \ldots, 1-d_n\right)$, which is equal to the image
 of $\mu^{-1}\left(d_1, \ldots, d_n\right)$ under the diffeomorphism above. 
 Again, the map ${\mathcal F}_{n, n-k}^{1-d} \to {\rm Gr}_{n-k}(\RR^n)$, $G\mapsto G^tG$ induces
 a homeomorphism
 $${\mathcal F}_{n, n-k}^{1-d}/{\rm O}(n-k) \simeq \nu^{-1}\left(1-d_1,\ldots, 1-d_n\right).$$

 By the assumption in Lemma \ref{lemeq} (c), let $\gamma: [0,1]\to {\mathcal F}_{n,k}^d$ be a continuous path with $\gamma(0) = F$ and $\gamma(1)=AF$.
 It descends to a path $\bar{\gamma} : [0,1] \to \mu^{-1}\left(d_1,\ldots, d_n\right)$, which is closed:
 $$\bar{\gamma}(0)=\bar{\gamma}(1).$$
 Consider  $\bar{\delta} : [0,1] \to \nu^{-1}\left(1-d_1,\ldots, 1-d_n\right)$,
 $$\bar{\delta}(s)\stackrel{\rm def}{=}{\rm I}_n-\bar{\gamma}(s),$$
 $s\in [0,1]$. Let $\delta : [0,1] \to {\mathcal F}_{n, n-k}^{1-d}$ be an arbitrary lift of $\tilde{\delta}$ relative to the ${\rm O}(n-k)$-bundle ${\mathcal F}_{n, n-k}^{1-d}\to \nu^{-1}\left(1-d_1,\ldots, 1-d_n\right)$. That is, for all $s\in [0,1]$ one has
 $${\rm I}_n-\gamma(s)^t\gamma(s) = \delta(s)^t\delta(s),$$ which we prefer to rewrite as
 $$\gamma(s)^t\gamma(s) + \delta(s)^t\delta(s)={\rm I}_n,$$
 and deduce from this that 
 the ($n\times n$) matrix
$$ \eta(s)\stackrel{\rm def}{=}\begin{pmatrix}
\gamma(s)  \\
{} \\
\hline\\
\delta(s) 
\end{pmatrix}$$
is in ${\rm O}(n)$.

 Since $\bar{\delta}(0)=\bar{\delta}(1)$, we must have $\delta(1)=B\delta(0)$, for some $B\in {\rm O}(n-k)$.
 By showing that $\det B=-1$ the proof will be finished. 
The idea is to consider the path in ${\rm O}(n)$ given by
$s\mapsto \eta(s)\eta(0)^t$, for $s\in [0,1]$. 
Its values at $s=0$ and $s=1$ are $I_n$ and
$$\begin{pmatrix}
A\gamma(0)   \\
{} \\
\hline\\
B\delta(0) 
\end{pmatrix} \begin{pmatrix} \gamma(0)^t \ \bigg| \delta(0)^t\end{pmatrix}=  \begin{pmatrix}
A & 0\\
0 &B
\end{pmatrix}$$
respectively.
Thus the latter matrix is in ${\rm SO}(n)$, hence $\det B =-1$. 
 \end{proof}
 
The second lemma is related to  \cite[Lemma 4.2]{CMS}. 
 
\begin{lem}  \label{switch} Let $n', n''$ be two integers, both at least equal to $k$ and such that $n'+n''=n$.
Let also\footnote{Neither the numbers $i_1, \ldots, i_{n'}$ nor $j_1,\ldots, j_{n''}$ are necessarily
in increasing order.}  $\{i_1,\ldots, i_{n'}\}$ and $\{j_1,\ldots, j_{n''}\}$ be a partition of the set 
$\{1, \ldots, n\}$.  
Take $(d'_1,\ldots, d'_{n'})\in \Delta_{n',k}$ and $(d''_1,\ldots, d''_{n''})\in \Delta_{n'',k}$  
along with the frames  $G'\in {\mathcal F}_{n', k}^{d'_1,\ldots, d'_{n'}}$ and $G''\in {\mathcal F}_{n'',k}^{d''_{1}, \ldots, d''_{n''}}$.
Let $\alpha$ and $\beta$ be two positive numbers such that $\alpha^2+\beta^2=1$.
 Consider the
vectors $g_1,\ldots, g_n \in \RR^k$ given by
 \begin{align*}
{}& g_{i_r}=\alpha g'_r, \ {\it for \ all \ } 1 \le r \le n',\\
{}& g_{j_s}=\beta g''_s,  \ {\it for \ all \ } 1 \le s \le n'',
\end{align*}
and also the numbers $$d_1=\|g_1\|^2,\ldots,  d_n=\|g_n\|^2.$$ Define  
$$F_0\stackrel{}{=}[g_1 \ | \ g_2 \ | \ \ldots \ | \ g_n]$$ which is in ${\mathcal F}^{d_1,\ldots, d_n}_{n, k}$. 
Assume that $d_{i_1}=d_{j_1}$. If $F_1$ denotes the $(d_1,\ldots, d_n)$-NTF obtained from $F_0$ by interchanging the columns 
$g_{i_1}$ and $g_{j_1}$, then there exists a continuous path in ${\mathcal F}^{d_1,\ldots, d_n}_{n, k}$
from $F_0$ to $F_1$.

\end{lem} 

\begin{proof} Pick $A\in {\rm SO}(k)$ such that $A g_{i_1}=g_{j_1}$.
Let $F_2$ be the $k\times n$ matrix whose columns of order $i_1, \ldots, i_{n'}$ are $Ag_{i_1}, \ldots, A g_{i_{n'}}$, respectively, 
the others being the same as those of $F_0$. Note that $AG'\in  {\mathcal F}_{n', k}^{d'_1,\ldots, d'_{n'}}$ and thus $F_2 \in {\mathcal F}^{d_1,\ldots, d_n}_{n, k}$.
 Since ${\rm SO}(k)$ is connected,
there exists a continuous path in ${\mathcal F}^{d_1,\ldots, d_n}_{n, k}$ from $F_0$ to $F_2$. 
 The components of $F_2$ of indices  $j_1, i_2, \ldots, i_{n'}$ form an $\alpha$-multiple of a  $(d_{i_1},\ldots, d_{i_{n'}})$-NTF
 and the components of indices $i_1, j_2, \ldots, j_{n''}$ a $\beta$-multiple of a $(d_{j_1},\ldots, d_{j_{n''}})$-NTF.
 Thus $F_2$ can be connected
with the frame whose components of indices  $j_1, i_2, \ldots, i_{n'}$ are
$A^tg_{j_1}, A^t(Ag_{i_2}), \ldots, A^t(Ag_{i_{n'}})$,  respectively, the others being the same as those of $F_2$.
Observe that the latter frame is just $F_1$.
\end{proof}
  
  Here are some situations when ${\mathcal F}_{n,k}^d$ is connected.
  
  \begin{prop}\label{prop:first}  Assume $n=2p$ and take 
   $$d=(d_1,\ldots, d_p, d_1, \ldots, d_p)\in \Delta_{n,k}$$ which satisfies assumption (\ref{hypothesis}).
   
   (a) Assume $p\ge k$. If $d_i \le \frac{1}{2}$ for all $1\le i \le p$, then $d$ is $k$-admissible.
   
   (b) Assume $p<k$. If $d_i \ge \frac{1}{2}$ for all $1\le i\le p$, then $d$ is $k$-admissible. 
 \end{prop} 
  
 \begin{proof} (a) Observe  that
  $(2d_1,\ldots, 2d_p)$ is in  $\Delta_{p,k}$ and thus there exists  
   $G\in {\mathcal F}^{2d_1,\ldots, 2d_p}_{p,k}$. Let  $D$ be the $k\times k$ matrix whose entries are all 0 except those on the diagonal, where we
   have $(1,  \ldots, 1, -1)$. Denote $G^-=DG$ and observe that this matrix is obtained from $G$ by adding a minus sign to all of its entries
   on the last row. The $k\times n$ matrix
$$F=  \frac{1}{\sqrt{2}} [ \ G \ | \ G^- \ ]$$     
  is in ${\mathcal F}^{d}_{n,k}$. We now use Lemma \ref{switch} several times successively, each time for
  $\alpha=\beta=\frac{1}{\sqrt{2}}$, to join $F$  with $\frac{1}{\sqrt{2}} [ \ G^- \ | \ G \ ]$ 
  by a continuous path in    ${\mathcal F}^{d}_{n,k}$. But the latter matrix is just $DF$, and since $\det D=-1$, Lemma \ref{lemeq} implies that
   ${\mathcal F}^{d}_{n,k}$ is connected.
   
   (b)  We have $p > n-k$ and from the fact that $d\in \Delta_{n,k}$ it follows that $1-d \in \Delta_{n,n-k}$. Moreover, since $d$ satisfies
   assumption (\ref{hypothesis}),  $1-d$ satisfies it too, this time relative to $n-k$. Thus, by (a),
   $1-d$ is $(n-k)$-admissible and by Lemma \ref{lem:DS}, $d$ is $k$-admissible. 
 \end{proof}

 \begin{prop} \label{prop:second} Assume $n=2p+1$ and take 
$$d=(d_1, \ldots, d_p, d_1, \ldots, d_p, d_{2p+1})\in \Delta_{n,k}$$
which satisfies assumption (\ref{hypothesis}). 

(a) Assume  
 $p\ge k-1$. Fix $(d'_1, \ldots, d'_p) \in \Delta_{p,k-1}$ which  is $(k-1)$-admissible. 
 If
 $$d_1>\frac{1}{2} d'_1, \ldots , d_p>\frac{1}{2}d'_p$$
 then  $d$ is $k$-admissible.
   
   (b) Assume $p<k-1$. Fix $(d'_1, \ldots, d'_p) \in \Delta_{p,n-k-1}$ which  is $(n-k-1)$-admissible. 
 If
 $$1-d_1>\frac{1}{2} d'_1, \ldots , 1-d_p>\frac{1}{2}d'_p$$
 then  $d$ is $k$-admissible. \end{prop}

 \begin{proof}
(a) There exists a $(d'_1, \ldots, d'_p)$-NTF, of the form $\tilde{G}=(\tilde{g}_{ij})_{1\le i \le k-1, 1\le j \le p}$.
   One can find a vector $g=(g_{k1}, \ldots, g_{k p})$ such that
   $$d_1=\frac{1}{2}d'_1+g_{k1}^2, \ldots , d_p= \frac{1}{2}d'_p+g_{kp}^2.$$
  Define 
   \begin{align*}{}& F=\begin{pmatrix}
\frac{1}{\sqrt{2}}\tilde{g}_{11} & \ldots & \frac{1}{\sqrt{2}}\tilde{g}_{1p} & -\frac{1}{\sqrt{2}}\tilde{g}_{11} & \ldots & -\frac{1}{\sqrt{2}}\tilde{g}_{1p} &0 \\
{}& \vdots & {}& {}& \vdots & {} & \vdots\\
\frac{1}{\sqrt{2}}\tilde{g}_{k-11} & \ldots & \frac{1}{\sqrt{2}}\tilde{g}_{k-1p} & -\frac{1}{\sqrt{2}}\tilde{g}_{k-11} & \ldots & -\frac{1}{\sqrt{2}}\tilde{g}_{k-1p} &0 \\
g_{k1} & \ldots & g_{kp} & g_{k1} & \ldots & g_{kp} & \sqrt{d_{2p+1}} 
\end{pmatrix}\\
{}& =\begin{pmatrix}
\frac{1}{\sqrt{2}}\tilde{G} & -\frac{1}{\sqrt{2}}\tilde{G} & 0 \\
g & g & \sqrt{d_{2p+1}} 
\end{pmatrix}.
\end{align*}
 Note that the  squared entries on the last row add up to
 $$2(d_1-\frac{1}{2}d'_1 + \ldots + d_p-\frac{1}{2}d'_p)+d_{2p+1} =k-(k-1)=1.$$
 Thus $F$  is a $d$-NTF.  
 Consider the following three diagonal matrices:
 \begin{itemize}
 \item $D$ of size $k\times k$ which has on the diagonal 
 $(1, \ldots, 1,  -1)$;
 \item $D'$ of size $k\times k$ which has on the diagonal 
 $(1, \ldots, 1, -1, -1)$; 
 \item $D''$ of size $(k-1)\times (k-1)$ which has on the diagonal 
 $(1, \ldots, 1,  -1)$.
 \end{itemize}
The goal is to show that $F$ can be joined with $DF$ by a path in ${\mathcal F}_{n,k}^d$. 
First, since $D'\in {\rm SO}(k)$, one can join $F$ with $D'F$. The latter matrix is
$$D'F= \begin{pmatrix}
\frac{1}{\sqrt{2}}D''\tilde{G} & -\frac{1}{\sqrt{2}}D''\tilde{G} & 0 \\
-g & -g & -\sqrt{d_{2p+1}} 
\end{pmatrix}.
$$ 
Since $(d'_1, \ldots, d'_p)$ is $(k-1)$-admissible, one can join 
$D''\tilde{G}$ with $\tilde{G}$ by a continuous path in ${\mathcal F}_{p, k-1}^{d'_1, \ldots, d'_p}$.
Thus we will be able to connect $D'F$ with 
 $$ \begin{pmatrix}
\frac{1}{\sqrt{2}}\tilde{G} & -\frac{1}{\sqrt{2}}\tilde{G} & 0 \\
-g & -g & -\sqrt{d_{2p+1}} 
\end{pmatrix},
$$ 
within ${\mathcal F}_{n,k}^d$. Since 
the latter matrix is $DF$, the proof is finished.

(b) From $p< k-1$ one deduces $p>n-k$, thus $p>n-k-1$. One uses item (a) for $n-k$ instead of $k$ and
$1-d$ instead of $d$. First note that $1-d \in \Delta_{n, n-k}$ and satisfies assumption (\ref{hypothesis}). 
It follows that $1-d$ is $n-k$-admissible.   To conclude, one uses Lemma \ref{lem:DS}, which implies that
$d$ is $k$-admissible. 
 \end{proof}

We will now illustrate Propositions \ref{prop:first} and \ref{prop:second}  by some examples.
  
  \subsection{The  case when $d_1= \cdots = d_n$} We first show how to deduce 
  the theorem of Cahill, Mixon, and Strawn \cite{CMS}, which asserts that for all $2\le k \le n-2$, the $n$-tuple 
  $(\frac{k}{n}, \ldots, \frac{k}{n})$ is $k$-admissible. For brevity we will express this by saying that {\it the pair $(n,k)$ is  admissible}.   
  First remark that in this case assumption (\ref{hypothesis}) is satisfied, see Ex.~\ref{example1}. We will gradually prove the general connectedness result
  by looking at various particular values of $(n, k)$. The case when $n$ is even follows by a direct application of Prop.~\ref{prop:first}.
  {\it From now on we assume that $n$ is odd}. The strategy is to use Prop.~\ref{prop:second} to 
  reduce the proof to the cases $k=2$ and $n=2k+1$, which in turn will be dealt with by using results from \cite{DS}.

 {\it Step 0.} $k=2$. We rely on \cite[Thm.~7.4]{DS}.
 
 {\it Step 1.} $n=2k+1$. Start with the matrix $G$ constructed in \cite[Ex.~3.2]{DS}. 
  It has size $k\times (k+1)$ and is of the form
  $$G=\begin{pmatrix}
  g_{11}&-g_{11}& 0 & \ldots& 0& 0\\
  g_{21}& g_{21}& g_{23}& \ldots & g_{2 k} & 0 \\
  g_{31}& g_{31}& g_{33}& \ldots & g_{3 k} & 0 \\
   \vdots& \vdots& \vdots& \ldots & \vdots & \vdots \\
    g_{k1}& g_{k1}& g_{k3}& \ldots & g_{k k} & -1 
    \end{pmatrix}.$$
    (In \cite{DS}, the matrix is denoted by $F$, it has size $n \times (n+1)$ and is a ``spherical tight frame",
which means that its columns are all of length 1 and
    the rows are pairwise orthogonal and of  length equal to $\sqrt{\frac{n+1}{n}}$.)

   Set $$F=\sqrt{\frac{k}{2k+1}} \ [ \ G \ | \ {\rm I}_k  \ ],$$ which is in ${\mathcal F}^{\frac{k}{2k+1}, \ldots, \frac{k}{2k+1}}_{2k+1,k}$.
   Denote by $A$ the diagonal $k\times k$ matrix ${\rm Diag}(-1, 1, \ldots, 1)$ and show that $F$ can be joined with $AF$
   by a continuous path in ${\mathcal F}^{\frac{k}{2k+1}, \ldots, \frac{k}{2k+1}}_{2k+1,k}$.   To this end, observe that 
   $$AF= \sqrt{\frac{k}{2k+1}} \ [ \ AG \ | \ A \ ]$$ is obtained from
   $F$ by interchanging the first two columns in $G$ and multiplying the $(k+2)$nd column by $-1$.
   
   We first join $F$ with $ \sqrt{\frac{k}{2k+1}} \ [ \ AG \ | \ {\rm I}_k \ ]$. This can be done in three steps, as follows.
   Use Lemma \ref{switch} for the pair of frames $\sqrt{\frac{k}{k+1}}G$ and $I_k$ together with
       $\alpha=\sqrt{\frac{k+1}{2k+1}}$ and $\beta=\sqrt{\frac{k}{2k+1}}$ in order to join $F$ with the frame obtained from it by interchanging the first and the $(k+2)$nd columns.
     By a similar argument, the latter frame can be joined with the one obtained from it by interchanging  the first and the second columns.
   Finally, the frame just mentioned can be joined with the one obtained from it by interchanging the  second and the $(k+2)$nd columns.
The last frame  is just $\sqrt{\frac{k}{2k+1}} \ [ \ AG \ | \ {\rm I}_k \   ]$.  

We now join $ \sqrt{\frac{k}{2k+1}} \ [  \ AG \ | \ {\rm I}_k \ ]$ with $ \sqrt{\frac{k}{2k+1}} \ [ \ AG \ | \ A \ ]$. We first consider the path  described by

$$\sqrt{\frac{k}{2k+1}}\left(\begin{array}{ccccccccccc}
  -g_{11} &g_{11} & 0 & \ldots   & 0& \cos t &  \sin t & 0 & 0 &\ldots  &  0  \\
  g_{21}& g_{21}& g_{23}& \ldots & g_{2 k} & 0 & 0 & 1 &  0 & \ldots & 0 \\
  g_{31}& g_{31}& g_{33}& \ldots & g_{3 k} & 0 & 0 & 0 & 1 &  \ldots & 0\\
   \vdots & \vdots & \vdots & \ldots & \vdots & \vdots & \vdots & \vdots & \vdots & \vdots & \vdots   \\
    g_{k1}& g_{k1}& g_{k3}& \ldots & g_{k k} & -\sin t & \cos t & 0 & 0 & \ldots & 1 \\
    \end{array} \right) 
 $$     
  for $t$ between $\frac{\pi}{2}$ and $\pi$ to join $ \sqrt{\frac{k}{2k+1}} \ [ \ AG \ | \ {\rm I}_k \ ]$ with

$$\sqrt{\frac{k}{2k+1}}\left(\begin{array}{ccccccccccc}
  -g_{11} &g_{11} & 0 & \ldots   & 0& -1 &  0 & 0 & 0 &\ldots  &  0  \\
  g_{21}& g_{21}& g_{23}& \ldots & g_{2 k} & 0 & 0 & 1 &  0 & \ldots & 0 \\
  g_{31}& g_{31}& g_{33}& \ldots & g_{3 k} & 0 & 0 & 0 & 1 &  \ldots & 0\\
   \vdots & \vdots & \vdots & \ldots & \vdots & \vdots & \vdots & \vdots & \vdots & \vdots & \vdots   \\
    g_{k1}& g_{k1}& g_{k3}& \ldots & g_{k k} & 0 & -1 & 0 & 0 & \ldots & 1 \\
    \end{array} \right) 
 $$  

Use Lemma \ref{switch} to join this with the frame obtained from it by interchanging the $(k+1)$st and the $(k+2)$nd columns.

 {\it Step 2.} {\it The general situation.} By Lemma \ref{lem:DS} and Step 1, we may assume that $n>2k+1$ (because if $n< 2k+1$ then actually $n<2k-1$, which implies $n> 2(n-k)+1$).
 Prop.~\ref{prop:second} will bring us in a position to apply  Steps 0 and 1. 
 Set $(n_0, k_0)=(n,k)$. One constructs a sequence $(n_i, k_i)$, where $k_i\ge 2$ and $n_i \ge 2k_i +1$, for all $i\ge 0$, which is defined inductively as follows:
 
\begin{itemize}
\item If $n_i$ is even or $n_i=2k_i+1$  or $k_i=2$ , the sequence stops. 
If $n_i$ is odd and $n_i >2k_i+1>5$, construct $(n_{i+1}, k_{i+1})$ as follows.
\item  Observe that
$$\frac{n_i-1}{2} > k_i-1.$$ Set 
 $$n_{i+1}:=\frac{n_i-1}{2}, \quad k'_{i+1}:=k_i-1.$$
 \item  If  $n_i \ge 4k_i-1$ then $n_{i+1} \ge 2k'_{i+1}+1$ and we set
 $$k_{i+1}:=k'_{i+1}.$$ 
 \item  If  $n_i <4k_i-1$ then $n_{i+1} < 2k'_{i+1}+1$ and we set
 $$k_{i+1}:=n_{i+1}-k'_{i+1}.$$ 
 \end{itemize}
 Note that in this way $n_{i+1}\ge 2k_{i+1}+1$.  
  
  The sequence will clearly stop at some point.    
 The crucial fact is that if $(n_{i+1}, k_{i+1})$ is admissible then also both $(n_{i+1}, k'_{i+1})$ and $(n_i, k_i)$ are admissible.
 This follows readily from Lemma \ref{lem:DS}, Prop.~\ref{prop:second} (a),  and the following  elementary inequality:
 $$\frac{k_{i}}{n_{i}}>\frac{1}{2}\cdot \frac{k'_{i+1}}{n_{i+1}}.$$  
  
We now present examples of $d\in \Delta_{n,k}$ which are $k$-admissible and whose entries are not all equal.
They are of the type described in Ex.~\ref{example2}. We consider separately the cases when $n$ is even and $n$ is odd.   
  
 \subsection{Other examples: the case when $n$ is even} \label{sub1}
  Concretely, say that $n=2p$ and pick an integer number $q$ such that
 $1\le q < p$. Set $k:=2q$. Consider $d$ like in Prop.~\ref{prop:first}, where $d_1=\cdots = d_q=:\alpha$ and $d_{q+1}=\cdots = d_p=:\beta.$
 That is,
 $$d=(\alpha, \ldots, \alpha, \beta, \ldots, \beta, \alpha, \ldots, \alpha, \beta, \ldots, \beta),$$
 where  $\alpha$ and $\beta$  occur each time $q$ and $p-q$ times respectively.  
As stated in Ex.~\ref{example2}, any real number $\beta$
such that 
\begin{align*}
\frac{1}{2p-2q}\le \beta \le \frac{q}{p}
\end{align*}
induces $d$  in $\Delta_{2p,2q}$ of the form above which satisfies assumption (\ref{hypothesis}).

We now use Prop.~\ref{prop:first}. Assume that $2q\le p$. One needs to impose that $\alpha \le \frac{1}{2}$, which is equivalent to
$$\beta\ge \frac{q}{2p-2q}.$$
Thus $d\in \Delta_{2p,2q}$, satisfies assumption (\ref{hypothesis}), and is $2q$-admissible provided that
\begin{align*}
\frac{q}{2p-2q}\le \beta \le \frac{q}{p}
\end{align*}
If now $2q >p$, a sufficient condition for $d$ to be $2q$-admissible is
  \begin{align*}
\frac{1}{2}\le \beta \le \frac{q}{p}.
\end{align*} 

\subsection{Other examples: the case when $n$ is odd} \label{sub2}
Say that $n=2p+1$ and consider again $q$ such that $1\le q <p$. 

For the beginning, take $k=2q+1$ and assume that $k<p$. This time consider $d$ like in Prop.~\ref{prop:second}, such that
$d_1=\cdots = d_q=d_{2p+1}=\alpha$ and $d_{q+1}=\cdots=d_p=\beta$.  Concretely,
$$d=(\alpha, \ldots, \alpha, \beta, \ldots, \beta, \alpha, \ldots, \alpha, \beta, \ldots, \beta, \alpha),$$
where  both enumerations of $\alpha$ and $\beta$  contain $q$ and $p-q$ elements respectively.   
Again by Ex.~\ref{example2}, out of any $\beta$
such that 
\begin{align*}
\frac{1}{2p-2q}\le \beta \le \frac{2q+1}{2p+1}
\end{align*}
one obtains $d$  in $\Delta_{2p+1,2q+1}$ as above which satisfies assumption (\ref{hypothesis}).
We apply  Prop.~\ref{prop:second} (a) for $d'_1=\cdots=d'_p=\frac{k-1}{p}=\frac{2q}{p}.$
Observe that
$$\frac{1}{2p-2q}<\frac{q}{p},$$
and consequently if
$$\frac{q}{p}\le \beta \le \frac{2q+1}{2p+1}$$
then $d$ is also $2q+1$-admissible. 

Let us now assume that $k=2q<p$ and also   that $d$ is of the form
$d_1=\cdots =d_q=\alpha$ and $d_{q+1}=\cdots = d_p=d_{2p+1}=\beta$, which means that
$$d=(\alpha, \ldots, \alpha, \beta, \ldots, \beta, \alpha, \ldots, \alpha, \beta, \ldots, \beta, \beta).$$
The condition which insures that $d$ is in  $\Delta_{2p+1,2q}$ and satisfies  (\ref{hypothesis}) is
  $$\frac{1}{2p-2q+1}\le \beta \le \frac{2q}{2p+1}.$$
  Again we apply  Prop.~\ref{prop:second} (a), this time for $d'_1=\cdots=d'_p=\frac{k-1}{p}=\frac{2q-1}{p}.$
  One can easily see that
  $$\frac{1}{2p-2q+1}<\frac{2q-1}{2p}.$$
  Thus if
$$\frac{2q-1}{2p} \le \beta \le \frac{2q}{2p+1}$$
then $d$ is  $2q$-admissible. 

Note that in both situations we had $k<p$. Examples with $k\ge p$ can now be easily
produced by means of Lemma \ref{lem:DS}.

To make the examples of admissible vectors obtained above more easily accessible, we present them in the following table.
In each of the four situations listed here, $d$ belongs to a 1-parameter family labelled by $\beta$, which lies in an interval.    
\begin{table}[ht]
\centering 
\begin{tabularx}{\textwidth}{l|X}
\hline 
{\bf The vector $d$ in terms of $\alpha$ and $\beta$} & {\bf The interval for $\beta$} \\
\hline
\hline 
For $n=2p, k=2q$, take & If $k\le p$, then $ \frac{q}{2p-2q}\le \beta \le \frac{q}{p}$. \\ 
$d=(\alpha, \ldots, \alpha, \beta, \ldots, \beta, \alpha, \ldots, \alpha, \beta, \ldots, \beta),$ where & {}  \\ 
$\alpha$  and $\beta$ repeat themselves  $q$ and $p-q$ times resp., &  If $k> p$, then $\frac{1}{2}\le \beta \le \frac{q}{p}$.   \\
and $q\alpha+(p-q)\beta=q$. & {} \\
\hline 
For $n=2p+1, k=2q+1$, such that $k<p$, take & {} \\ 
$d=(\alpha, \ldots, \alpha, \beta, \ldots, \beta, \alpha, \ldots, \alpha, \beta, \ldots, \beta, \alpha),$ where & $\frac{q}{p}\le \beta \le \frac{2q+1}{2p+1}$  \\ 
 $\alpha$ and $\beta$ repeat themselves  $q$ and $p-q$ times resp., &  {}   \\
and $(2q+1)\alpha+(2p-2q)\beta=2q+1$. & {} \\
\hline 
For $n=2p+1, k=2q$, such that $k<p$, take & {} \\ 
$d=(\alpha, \ldots, \alpha, \beta, \ldots, \beta, \alpha, \ldots, \alpha, \beta, \ldots, \beta, \beta),$ where & $\frac{2q-1}{2p}\le \beta \le \frac{2q}{2p+1}$  \\ 
 $\alpha$ and $\beta$ repeat themselves  $q$ and $p-q$ times resp., &  {}   \\
and $2q\alpha+(2p-2q+1)\beta=2q$. & {} \\
\hline
\end{tabularx}
\end{table}

\subsection{Relations with  spaces of polygons in Euclidean plane.} \label{poly}

 Fix an integer number $n\ge 3$. An $n$-gon in $\RR^2$ is a sequence of $n$ vectors 
 $e_1, \ldots, e_n \in \RR^2$ such that $e_1+ \cdots + e_n=0$. The vectors are  actually the sides of the $n$-gon.
Fix non-negative real numbers $r_1, \ldots, r_n$   such that
$$r_1+\cdots+ r_n=1.$$ Let $\tilde{\mathcal P}(r_1, \ldots, r_n)$ be  
the set consisting  
of all $n$-gons with side-lengths
$$\|e_j\|=r_j, \ 1\le j \le n.$$
This  space, along with its quotient modulo the  group of orientation-preserving rotations ${\rm SO}(2)$, is by now fairly well understood. 
For example, as a subspace of $(\RR^2)^n$, 
 it becomes a topological space and as such it has been  
investigated in \cite{KM}, \cite{HK1}, and \cite{HK2}. 
Among others, it was shown in \cite{KM} that $\tilde{\mathcal P}(r_1, \ldots, r_n)$ is not connected if and only if 
there exist three pairwise distinct indices $i,j,$ and $k$ such that
$$r_i+r_j > \frac{1}{2},  \  r_j+r_k >\frac{1}{2},  {\rm and}  \ r_i+r_k >\frac{1}{2}.$$
 Important for us is the relationship between $\tilde{\mathcal P}(r_1, \ldots, r_n)$ and a certain space of normalized tight frames,
 which can be described as follows. First, identify $\RR^2$ with the field $\CC$ of complex numbers in the usual way. 
One can show that the map ${\mathcal F}_{n, 2}^{d_1,\ldots, d_n} \to \tilde{\mathcal P}(\frac{1}{2}d_1, \ldots, \frac{1}{2}d_n)$
given by 
$$[ \ z_1 \ | \ \ldots \ | \ z_n \ ] \mapsto \frac{1}{2}(z_1^2 , \ldots, z_n^2)$$
is well defined and surjective, see \cite[Prop.~5.1]{DS} and the references indicated in that paper.
 We conclude the discussion with the following criterion: if one can find three pairwise distinct indices $i,j$, and $k$ such that
 $$ d_i+d_j > 1,  \  d_j+d_k >1,  {\rm and}  \ d_i+d_k >1,$$
 then ${\mathcal F}_{n, 2}^{d_1,\ldots, d_n}$ is not connected.

On the other hand, by specializing the two criteria stated in Propositions \ref{prop:first} and \ref{prop:second} to the case $k=2$, the space
 ${\mathcal F}_{n, 2}^{d_1,\ldots, d_n}$ is connected under certain assumptions. One of them is the condition given by eq.~(\ref{hypothesis}).
 This condition immediately implies that for any two  distinct indices $i$ and $j$ one has
 $$d_i+d_j \le 1,$$
 which is in accordance with the observation made above.   

\begin{ex} The  necessary condition for ${\mathcal F}_{n, 2}^{d_1,\ldots, d_n}$  not to be connected mentioned above is not sufficient.
Just like in Example \ref{exa}, consider again the case when $n=4$, $k=2$, and $(d_1, d_2,d_3,d_4)=(1,1,0,0)$, for which that condition is
clearly not satisfied.  
On the other hand, the elements of ${\mathcal F}_{4, 2}^{1,1,0,0}$ are matrices of the type
$[ \ A \ | \ 0 \ ]$ where both $A$ and $0$ are $2\times 2$ matrices, the latter having all entries equal to 0 and the former being
orthogonal. Thus $ {\mathcal F}_{4, 2}^{1,1,0,0}$ is homeomorphic to the orthogonal group ${\rm O}(2)$, and consequently is not connected.
Observe however that both $\mu^{-1}(1,1,0,0)$ and $\tilde{\mathcal P}(\frac{1}{2},\frac{1}{2},0,0)$ are connected, see Example \ref{exa} and the theorem of Kapovich and Millson mentioned above.     
\end{ex}

\end{document}